\documentclass[11pt, english]{amsart}

\usepackage{amsmath,amssymb,enumerate}

\usepackage[T1]{fontenc}
\usepackage[all]{xy}

\usepackage{babel}
\usepackage{amstext}
\usepackage{amsmath}
\usepackage{amsfonts}
\usepackage{latexsym}
\usepackage{ifthen}

\usepackage{xypic}
\xyoption{all}
\newcommand{\rk}{{\rm rk}}

\newcommand{\sK}{{\mathcal K}}

\newtheorem{lemma1}{}[section]

\newenvironment{lemma}{\begin{lemma1}{\bf Lemma.}}{\end{lemma1}}

\newenvironment{theorem}{\begin{lemma1}{\bf Theorem.}}{\end{lemma1}}

\newenvironment{corollary}{\begin{lemma1}{\bf Corollary.}}{\end{lemma1}}

\newenvironment{assumption}{\begin{lemma1}{\bf Assumption.}}{\end{lemma1}}

\newenvironment{remark*}{{\bf Remark.}}{}
\newenvironment{example*}{{\bf Example.}}{}
\newenvironment{assumption*}{{\bf Assumption.}}{}

\newcommand{\C}{\ensuremath{\mathbb{C}}}
\newcommand{\N}{\ensuremath{\mathbb{N}}}
\newcommand{\PP}{\ensuremath{\mathbb{P}}}

\newcommand{\holom}[3]{\ensuremath{#1:#2  \rightarrow #3}}
\newcommand{\fibre}[2]{\ensuremath{#1^{-1} (#2)}}

\makeatletter
\ifnum\@ptsize=0  \fi \ifnum\@ptsize=2  \fi 

\newcommand\sO{{\mathcal O}}

\DeclareMathOperator*{\sing}{sing}

\DeclareMathOperator*{\nons}{nons}

\DeclareMathOperator*{\Plog}{\Omega_{\PP^n}(\log D)}
\DeclareMathOperator*{\Plogone}{\Omega_{\PP^n}(\log D) \otimes \sO_{\PP^n}(1)}

\title{Totally invariant divisors of endomorphisms of projective spaces} 
\date{\today}

\subjclass[2010]{14J70, 14F10, 14J40}
\keywords{endomorphism, totally invariant divisor}

\author{Andreas H\"oring}

\address{Andreas H\"oring, Laboratoire de Math{\'e}matiques J.A. Dieudonn{\'e},
UMR 7351 CNRS, Universit{\'e} de Nice Sophia-Antipolis, 06108 Nice Cedex 02, France        
}
\email{hoering@unice.fr}

\begin{document}

\begin{abstract}
Totally invariant divisors of endomorphisms
of the projective space are expected to be always unions of linear spaces. 
Using logarithmic differentials we establish a lower bound for the degree of the non-normal locus of a totally invariant divisor. As a consequence we prove the linearity of totally invariant divisors for $\PP^3$.
\end{abstract}

\maketitle

\vspace{-0.2cm}
\section{Introduction}

An endomorphism of a complex projective variety $X$ is a finite morphism $\holom{f}{X}{X}$
of degree at least two. A totally invariant subset of $f$ is a subvariety $D \subset X$
such that we have a set-theoretic equality $\fibre{f}{D}=D$.
The projective space $X=\PP^n$ admits many endomorphisms (simply take $n+1$ homogeneous
polynomials of degree $m$ without a common zero), and it is an interesting problem to understand
their dynamics \cite{FS94}. A well-known conjecture claims that totally invariant subvarieties
of endomorphisms $\holom{f}{\PP^n}{\PP^n}$ are always linear 
subspaces. This conjecture\footnote{This statement is claimed in \cite{BCS04}, but the proof has a gap.} is known for divisors of degree $n+1$ \cite[Thm.2.1]{HN11} and smooth hypersurfaces of any degree.
In fact, by results of Beauville \cite[Thm.]{Bea01}, Cerveau-Lins Neto \cite{CL00}
and Paranjape-Srinivas \cite[Prop.8]{PS89} a smooth hypersurface $D$ of degree at least two does not admit an endomorphism, {\em in particular} it is not a totally invariant subset of $\holom{f}{\PP^n}{\PP^n}$. However there are examples of singular normal hypersurfaces $D \subset \PP^n$ of degree $n$ that admit an endomorphism
$\holom{g}{D}{D}$ \cite[Ex.1.9]{Zha14}. One should thus ask if $g$ is induced by an endomorphism
of the projective space. The main result of this paper is a negative answer to this question:

\begin{theorem} \label{theoremmain}
Let $\holom{f}{\PP^n}{\PP^n}$ be an endomorphism of degree at least two, and let $D \subset \PP^n$ be a prime divisor
of degree $d \geq 2$ that is totally invariant. Denote by $Z \subset D$ the non-normal locus of $D$. Then we have
\begin{equation} \label{theestimate}
\deg(Z) > (d-1)^2 - \frac{n(n-1)}{2}.
\end{equation}
In particular if $d \geq 1 + \sqrt{\frac{n(n-1)}{2}}$, then $D$ is not normal.
\end{theorem}

Note that if $d=n$, then the inequality \eqref{theestimate} simplifies to
$$
\deg(Z) > \frac{1}{2} (n-2) (n-1).
$$
However, by a well-known result about singularities of irreducible plane curves \cite[3.8]{Fis01}, one has 
$\deg(Z) \leq \frac{1}{2} (n-1)(n-2)$. Thus an {\em irreducible} divisor $D$ of degree $n$ is not totally invariant.
This observation significantly improves \cite[Thm1.1]{Zha13}, combined with \cite[Thm.1.5(5) (arXiv version)]{NZ10} we obtain:

\begin{corollary} \label{corollarythreespace}
Let $\holom{f}{\PP^3}{\PP^3}$ be an endomorphism, and let $D \subset \PP^n$ be a prime divisor
that is totally invariant. Then $D$ is a hyperplane.
\end{corollary}

\begin{center}
{\bf Notation and terminology.} 
\end{center}

We work over the complex field $\C$. Let $\holom{f}{\PP^n}{\PP^n}$ be an endomorphism, and let $D \subset \PP^n$ be a totally
invariant prime divisor. Then (e.g. by \cite[Lemma 2.5]{a20}) there exists a unique effective divisor $R$ 
such that the logarithmic ramification formula
$$
K_{\PP^n} + D = f^* (K_{\PP^n} + D) + R
$$
holds, and we call $R$ the logarithmic ramification divisor. Since $\rho(\PP^n)=1$ one easily deduces that $d := \deg D \leq \deg(-K_{\PP^n}) = n+1$.

Given a locally free sheaf $E \rightarrow X$ over some manifold $X$ and $x \in X$ a point,
we denote by $E_x$ the $\C$-vector space $E \otimes \sO_X/m_x$ where $m_x \subset \sO_X$ is the ideal
sheaf of $x$. If $\alpha: E_1 \rightarrow E_2$ is a morphism of sheaves between locally free sheaves $E_1$ and $E_2$,
we denote by $\alpha_x: E_{1,x} \rightarrow E_{2,x}$ the linear map induced between the vector spaces.

\section{The sheaf of logarithmic differentials} \label{sectionlogdiff}
We consider the complex projective space $\PP^n$ of dimension $n \geq 2$.

\begin{assumption} \label{assume}
In this whole section we denote by $D \subset \PP^n$ a prime divisor of degree $d \geq 2$. We suppose that  there exists a subset $W \subset \PP^n$ of codimension at least three such that
$D \setminus W$ has at most normal crossing singularities. 
\end{assumption}

\subsection{Definition and Chern classes}
Since $D$ has normal crossing singularities in codimension two, 
the sheaf of logarithmic differentials
in the sense of Saito \cite{Sai80}  
and the sheaf of logarithmic differentials in the sense of Dolgachev \cite[Defn.2.1]{Dol07} coincide by
\cite[Cor.2.2]{Dol07}, we will denote this sheaf by $\Omega_{\PP^n}(\log D)$. The sheaf $\Omega_{\PP^n}(\log D)$
is reflexive (it is defined as a dual sheaf \cite[p.36, line -4]{Dol07}) and locally free in the points where $D$
has normal crossing singularities.
By \cite[(2.8)]{Dol07} 
there exists a residue exact sequence
\begin{equation} \label{definelog}
0 \rightarrow \Omega_{\PP^n} \rightarrow \Omega_{\PP^n}(\log D) \rightarrow \nu_*(\sO_{\tilde D}) \rightarrow 0,
\end{equation}
where $\holom{\nu}{\tilde D}{D}$ is the normalisation\footnote{The statement in \cite[(2.8)]{Dol07}
is for a desingularisation, but since $\pi_*(\sO_{D''}) = \sO_{D'}$ for any birational morphism $\holom{\pi}{D''}{D'}$
between normal varieties, the statement holds for the normalisation.}.  

Our goal is to compute the first and second Chern class of the sheaf $\Omega_{\PP^n}(\log D)$.
Recall first that
\begin{equation} \label{cherndivisor}
c_1(\sO_{D}) = D, \ c_2(\sO_{D})= D^2.
\end{equation}
Denote by $Z \subset D$ the non-normal locus of $D$. 
Since $D$ is Cohen-Macaulay, we know by Serre's criterion that $Z \subset \PP^n$ is empty or a projective set of pure dimension $n-2$.
We have an exact sequence
\begin{equation} \label{normalisationsequence}
0 \rightarrow \sO_D \rightarrow \nu_* \sO_{\tilde D} \rightarrow \sK \rightarrow 0,
\end{equation}
where $\sK$ is a sheaf with support on $Z$.
Since $D$ has normal crossings on $D \setminus W$ the restriction of
\eqref{normalisationsequence} to $D \setminus W$ is
\begin{equation} \label{normalisationsequence2}
0 \rightarrow \sO_{D \setminus W} \rightarrow \nu_*(\sO_{\tilde D}) \otimes \sO_{D \setminus W} 
\rightarrow \sO_{Z \cap (D \setminus W)} \rightarrow 0.
\end{equation}
Since $Z$ is empty or of pure dimension $n-2$ and $W$ has codimension at least three in $\PP^n$,
we see that $W$ does not contain any irreducible component of $Z$.
The second Chern class $c_2(\nu_*(\sO_{\tilde D}))$ is determined by intersecting with the class of 
a general linear $2$-dimensional subspace $P \subset \PP^n$. Since $P$ is disjoint from $W$,  
the sequence \eqref{normalisationsequence2} combined with \eqref{cherndivisor} yields
\begin{equation} \label{chernclassnu}
c_1(\nu_*(\sO_{\tilde D})) = D, \ c_2(\nu_*(\sO_{\tilde D}))= D^2 - [Z].
\end{equation}
Recall now that $c_1(\Omega_{\PP^n}) = (n+1) H, \ c_2(\Omega_{\PP^n}) = \frac{n(n+1)}{2} H^2$ where $H$ is the hyperplane class.
Then the exact sequence \eqref{definelog} combined with \eqref{chernclassnu} yields
$$
c_2(\Plog) = \left( \frac{(n+1) (n-2d)}{2}+d^2 \right) H^2 - [Z].
$$
Thus if we twist by $\sO_{\PP^n}(m)$ we obtain that
\begin{multline} \label{twistm}
c_2(\Plog \otimes \sO_{\PP^n}(m)) = \left( \frac{(n+1) (n-2d)}{2}+d^2 \right) H^2 - [Z]
\\ - (n-1) (n+1-d) m H^2 + \frac{n(n-1)}{2} m^2 H^2.
\end{multline}
For $m=1$ this formula simplifies to
\begin{equation} \label{twistone}
c_2(\Plog \otimes \sO_{\PP^n}(1)) = (d-1)^2 H^2 - [Z].
\end{equation}

\subsection{Global sections of $\Plogone$}
We now choose homogeneous
coordinates $X_0, \ldots, X_n$ on $\PP^n$.
Since $D \subset \PP^n$ is a prime divisor of degree $d \geq 2$, we have
$$
H^0(D, \sO_D(1)) =
\langle
X_0|_D, X_1|_D, \ldots, X_n|_D
\rangle,
$$
and, for simplicity's sake, we denote by $X_0|_D, X_1|_D, \ldots, X_n|_D$ also their images in 
$H^0(D,  \nu_*(\sO_{\tilde D}))$ under the natural inclusion $H^0(D, \sO_D) \subset H^0(D,  \nu_*(\sO_{\tilde D}))$.
By Bott's theorem we have $H^1(\PP^n, \Omega_{\PP^n}(1))=0$, so the cohomology
sequence associated to the sequence \eqref{definelog} twisted by $\sO_{\PP^n}(1)$
shows that $X_0|_D, X_1|_D, \ldots, X_n|_D$ lift to global sections of $\Plogone$.
In fact if we denote by $f$ an irreducible homogeneous polynomial 
defining the hypersurface $D$, these global sections
can be written in homogeneous coordinates as
\begin{equation} \label{theforms}
\frac{d (X_0 \cdot f)}{f}, \frac{d (X_1 \cdot f)}{f}, \ldots, \frac{d (X_n \cdot f)}{f}.
\end{equation}

The following elementary lemma is fundamental for our proof.

\begin{lemma} \label{lemmaglobalgeneration}
Under the Assumption \ref{assume}, let
$$
\alpha: \sO_{\PP^n}^{\oplus n+1} \rightarrow \Plogone
$$
be the morphism of sheaves defined by the global sections \eqref{theforms}. Then $\alpha$
is surjective on $\PP^n \setminus D_{\sing}$. If $x \in D_{sing}$ is a point such that in local analytic
coordinates $u_1, \ldots, u_n$ around $x$ the hypersurface $D$ is given by $u_1 \cdot u_2=0$, the linear map
$$
\alpha_x : (\sO_{\PP^n}^{\oplus n+1})_x \rightarrow (\Plogone)_x 
$$
has rank at least $n-1$.
\end{lemma}

For the proof recall the well-known local description of logarithmic differentials in the points where $D$
is a normal crossings divisor: fix a point $x \in D$ and let $u_1, \ldots, u_n$ be holomorphic coordinates
in an analytic neighbourhood of $x$. If $D$ is given by $u_1=0$ in these coordinates (so $x \in D_{\nons}$), then
$\Omega_{\PP^n}(\log D)$ is locally generated by
$$
\frac{d u_1}{u_1}, d u_2, \ldots, d u_n.
$$
If $D$ is given by $u_1 \cdot u_2=0$ a set of local generators is
$$
\frac{d u_1}{u_1}, \frac{d u_2}{u_2}, d u_3, \ldots, d u_n.
$$

\begin{proof}[Proof of the first statement.]
We prove the statement for $x \in D \setminus D_{\sing}$, the (easier)
case $x \in \PP^n \setminus D$ is left to the reader. 
Up to linear coordinate change we can suppose
that $x=(1: 0: \ldots: 0)$.
The affine set 
$U_0 := \{
x \in \PP^n \ | \ x_0 \neq 0
\}$ 
is isomorphic to $\C^n$ under the isomorphism 
$$
(X_0: \ldots : X_n) \ \mapsto \ (\frac{X_1}{X_0}, \ldots,  \frac{X_n}{X_0})=(Y_1, \ldots, Y_n).
$$
In this affine chart the forms 
\eqref{theforms} can be written as
\begin{equation} \label{affineexpression}
\frac{df_b}{f_b}, \frac{Y_1 df_b}{f_b} +  d Y_1, \ldots, \frac{Y_n df_b}{f_b} +  d Y_n,
\end{equation}
where $f_b(Y_1, \ldots, Y_n):=f(1, Y_1, \ldots, Y_n)$ is the deshomogenisation of $f$.
Since $x \in D$ is a smooth point one of the partial derivatives
$\frac{\partial f_b}{\partial Y_i}(x)$ is non-zero, so up to renumbering the coordinates
$Y_1, \ldots Y_n$ we can suppose
that $\frac{\partial f_b}{\partial Y_1} (x) \neq 0$. 
Thus $f_b, Y_2, \ldots, Y_n$ form a set of holomorphic coordinates around $x$ and
$$
\frac{df_b}{f_b}, d Y_2, \ldots, d Y_{n}
$$
is a set of generators for $(\Plogone)|_{U_0}$ in a neighbourhood of $x$.
Yet in the point $x=(0, \ldots, 0)$ the global sections \eqref{affineexpression} are equal
to $\frac{df_b}{f_b}, d Y_1, \ldots, d Y_{n}$, so they contain this generating set.

{\em Proof of the second statement.} Up to linear coordinate change we can suppose
that $x=(1: 0:\ldots: 0)$ and as before we consider the affine chart $U_0 \simeq \C^n, Y_i = \frac{X_i}{X_0}$
and the expression $\eqref{affineexpression}$ of the global sections in these affine coordinates.
Up to renumbering we can suppose that $u_1, u_2, Y_3, \ldots, Y_n$ are coordinates
in an analytic neighbourhood of $(0, \ldots, 0) \in \C^n$. Thus $(\Plogone)|_{U_0}$
is generated in a neighbourhood of the origin by
$$
\frac{d u_1}{u_1}, \frac{d u_2}{u_2}, d Y_3, \ldots, d Y_n.
$$
The logarithmic forms $\frac{Y_i df_b}{f_b} +  d Y_i$ are equal to $d Y_i$ in the origin,
so they
generate the subspace 
$$
\langle d Y_3, \ldots, d Y_n \rangle \subset (\Plogone)_x.
$$
In the coordinates $u_1, u_2, Y_3, \ldots, Y_n$ the polynomial $f_b$ is equivalent to $u_1 \cdot u_2$,
and
$$
\frac{d (u_1 \cdot u_2) }{u_1 u_2} = \frac{d u_1}{u_1} + \frac{d u_2}{u_2} 
$$
is a non-zero element of $(\Plogone)_x$
which is not in the $(n-2)$-dimensional subspace $\langle d Y_3, \ldots, d Y_n \rangle$. Thus the global sections generate
a subspace of dimension at least $n-1$.
\end{proof}

\section{Proof of the main theorem}

The proof of Beauville's result \cite[Thm.]{Bea01} on endomorphisms of smooth hypersurfaces
$D \subset \PP^n$ is based on the fact that a global section of $\Omega_{X}(2)$ with isolated
zeroes maps under the tangent map to a global section of $\Omega_{X}(2m)$
which still has isolated zeroes \cite[Lemma 1.1]{ARV99}.
The following technical statement gives an analogue for our setting:

\begin{lemma} \label{lemmasurfaceinequality}
Let $S$ be a smooth projective surface, and let $E_1$ be a vector bundle on $S$ of rank $n \geq 2$.
Suppose that there exists a linear subspace $V \subset H^0(S, E_1)$ such that
$\dim V>\rk E_1$ and the evaluation morphism
$$
ev: V \otimes \sO_S \rightarrow E_1
$$
is surjective in the complement of a finite set $Z_S \subset S$. Suppose also that for every 
point $x \in Z_S$ the linear map
$$
ev_x : (V \otimes \sO_S)_x \rightarrow E_{1, x}
$$
has rank at least $n-1$. 

Suppose that there exists a vector bundle $E_2$ on $S$ of rank $n$ and an injective morphism of sheaves
$$
\varphi: E_1 \rightarrow E_2
$$
such that the following holds:
\begin{enumerate}
\item[(a)] The linear map $\varphi_x: E_{1,x} \rightarrow E_{2,x}$ has rank at least $n-2$ in every point $x \in S$. The set
$B_S$ where $\rk(\varphi_x)=n-2$ is finite.
\item[(b)] Denote by $R_S \subset S$ the closed set such that $\rk(\varphi_x)<n$. Then $R_S$
is disjoint from $Z_S$.
\end{enumerate}
Then we have $c_2(E_1) \leq c_2(E_2)$.
\end{lemma}

\begin{proof} Denote by $|V|$ the projective space associated to the vector space $V$.
Consider the projective set
$$
B := \{
(x, \sigma) \in X \times |V| \ | \ \varphi(ev(\sigma(x)))=0
\},
$$
and denote by $\holom{p_1}{B}{X}$ and \holom{p_2}{B}{|V|} the natural projections.
If $x \in B_S \subset R_S$, then $x \not\in Z_S$ by hypothesis $(b)$. Thus $(\varphi \circ ev)_x$ 
has rank $n-2$ and $\dim \fibre{p_1}{x} = \dim V-n+1$.
Analogously if $x \in R_S \setminus B_S$ (resp. $x \in Z_S$), then $\dim \fibre{p_1}{x} = \dim V-n$.
Finally for $x \in S \setminus (R_S \cup Z_S)$ we obviously have $\dim \fibre{p_1}{x} = \dim V-n-1$.
Thus we see that all the irreducible components of $B$ have dimension at most $\dim V-n+1$.

We will now argue by induction on the rank $n$. 

{\em Start of the induction: $n=2$.} Then all the irreducible components have dimension at most $\dim V-1=\dim |V|$,
so the general fibre of $p_2$ is finite or empty. Hence for a general $\sigma \in |V|$,
we have an induced section 
$$
\sO_S \stackrel{\sigma}{\rightarrow} E_1 \stackrel{\varphi}{\rightarrow} E_2
$$
of $E_2$ which vanishes at most in finitely many points (so it computes $c_2(E_2)$). In particular 
the section $\sO_S \stackrel{\sigma}{\rightarrow} E_1$ vanishes at most in finitely many points and clearly $c_2(E_1) \leq c_2(E_2)$.

{\em Induction step: $n>2$.} In this case all the irreducible components have dimension at most $\dim V-1 < \dim |V|$,
so the general $p_2$-fibre is empty. Thus a general $\sigma \in |V|$ defines a morphism
$$
\sO_S \stackrel{\sigma}{\rightarrow} E_1 \stackrel{\varphi}{\rightarrow} E_2
$$
that does not vanish, hence it defines a trivial subbundle of both $E_2$ and $E_1$. In particular
the quotients $E_2/\sO_S$ and $E_1/\sO_S$ are locally free and it is easy to check that
the space of global sections $V/\C \sigma$ and the induced map $\bar \varphi: E_1/\sO_S \rightarrow E_2/\sO_S$
still satisfy the conditions of the lemma. Since
$c_2(E_i)=c_2(E_i/\sO_S)$ we can conclude.
\end{proof}

\begin{proof}[Proof of Theorem \ref{theoremmain}]
Since $\PP^n$ has Picard number one, the endomorphism $f$ is polarised, i.e. we have
$f^* H \equiv mH$ for some $m \in \N$ and $H$ the hyperplane class.
Since $D$ is totally invariant, we know by \cite[Cor.3.3]{a20} (cf. also \cite[Prop.2.4]{HN11}) 
that the pair $(\PP^n, D)$ is log-canonical. Since $D$ is Cohen-Macaulay its non-normal locus
$Z$ has pure dimension $n-2$ and every irreducible component of $Z$ is an lc centre
of the pair $(X, D)$. Thus we know by \cite[Cor.3.3]{a20} that (up to replacing $f$ by some
iterate $f^l$) every irreducible component of $Z$ is totally invariant and 
not contained in the logarithmic
ramification divisor $R$. Since $D$ is totally invariant for any iterate $f^l$, we can suppose from now
on that these properties hold for $f$.

Since the pair $(X, D)$ is log-canonical there exists 
a subset $W \subset \PP^n$ of codimension at least three such that
$D \setminus W$ has at most normal crossing singularities. Thus we can use the logarithmic
cotangent sheaf $\Plog$ introduced in Section \ref{sectionlogdiff}. Since $D$ is a totally invariant
divisor, the tangent map
$$
df : f^* \Omega_{\PP^n} \rightarrow \Omega_{\PP^n}
$$
induces an injective morphism of sheaves 
$$
df_{\log}: f^* \Plog \rightarrow \Plog.
$$
Let $P \subset \PP^n$ be a general $2$-dimensional linear subspace, and $S: = \fibre{f}{P}$
its preimage. Then $S$ is a smooth surface, and we claim that
$$
\varphi : f^*(\Plog \otimes \sO_{\PP^n}(1)) \otimes \sO_S \rightarrow 
\Plog \otimes \sO_{\PP^n}(m) \otimes \sO_S
$$
satisfies the conditions of Lemma \ref{lemmasurfaceinequality}.

{\em Proof of the claim.}
Consider the $n+1$-dimensional subspace $V \subset H^0(\PP^n, \Plog \otimes \sO_{\PP^n}(1))$
defined by the global sections \eqref{theforms}.
By Lemma \ref{lemmaglobalgeneration} the evaluation morphism is surjective in the complement of the singular
locus $D_{\sing}$, and if $x \in Z$ is a general point, it has rank at least $n-1$. Since $P$ is general of dimension two, the intersection $P \cap D_{\sing}$ consists only of general points of $Z$, so if we denote by
$$
ev_S : f^* (V \otimes \sO_{\PP^n}) \otimes \sO_S \rightarrow f^* (\Plog \otimes \sO_{\PP^n}(1)) \otimes \sO_S
$$
the restriction of the (pull-back of the) evaluation morphism to $S$ it is surjective in the complement
of the finite set $Z_S := \fibre{f}{P \cap Z}$ and has rank at least $n-1$ in the points of $Z_S$.
Since $Z$ is totally invariant, the finite set $Z_S$ is contained in $Z \cap S$. Since $Z$ is not contained
in the logarithmic ramification divisor $R$ and $P$ is general, the intersection $Z \cap R \cap S$ 
is empty. This shows that the sets $R_S := R \cap S$ and $Z_S$ are disjoint.

Thus we are left to show that $\rk \varphi_x \geq n-2$ for every $x \in S$ and the set $B_S$ where
equality holds is finite. For the tangent map $df$ this is well-known: if $W \subset \PP^n$ is 
a variety of dimension $d$ and $x \in W$ is a general point, the finite map $W \rightarrow f(W)$ 
is \'etale in $x$, in particular the tangent map $df$ has rank at least $\dim W$ in $x$.
This shows that the sets 
$$
\{
x \in \PP^n  \ | \ \rk \ df_x \leq n-k
\}
$$ 
have codimension at least $k$ in $\PP^n$. Since $\Omega_{\PP^n}$ and $\Plog$ identify in the complement
of $D$ we are thus left to consider points of $D$. Yet if
$x \in D_{\nons}$ (resp. $x \in Z$ general) the vector space $\Plog_x$ contains a 
linear subspace that is naturally isomorphic to $\Omega_{D,x}$
(resp. $\Omega_{Z,x}$), so we can reduce to the case of the tangent map of $f|_D$ (resp. $f|_Z$).
This proves the claim.

We can now finish the proof by comparing the Chern classes.
Since $f^* H \equiv m H$ we have $[S] = m^{n-2} H^{n-2}$ and $f^* [Z] = m^2 (\deg Z) H^2$. Thus it follows from \eqref{twistone} that
$$
c_2(f^* (\Plog \otimes \sO_{\PP^n}(1)) \otimes \sO_S) = ((d-1)^2 - \deg Z) m^n.
$$
By \eqref{twistm} and Lemma \ref{lemmasurfaceinequality} this is less or equal than
\begin{multline}
c_2(\Plog \otimes \sO_{\PP^n}(m) \otimes \sO_S) = \left(\frac{(n+1) (n-2d)}{2}+d^2 - \deg Z \right) m^{n-2} 
\\ - (n-1) (n+1-d) m^{n-1} + \frac{n(n-1)}{2} m^n.
\end{multline}
Since we can replace $f$ by some iterate the inequality holds for all sufficiently
divisible $m \in \N$. Thus by considering only the terms of order $m^n$ we obtain
\begin{equation} \label{almostdone}
(d-1)^2 - \deg Z \leq \frac{n(n-1)}{2}.
\end{equation}
This inequality is always strict since otherwise we obtain
$$
0 \leq \left(\frac{(n+1) (n-2d)}{2}+d^2 - \deg Z \right) m^{n-2} 
- (n-1) (n+1-d) m^{n-1}
$$
for all sufficiently divisible $m \in \N$. Now recall that $d \leq n+1$ and $d=n+1$ is excluded since we suppose that $D$ is a prime divisor \cite[Thm.2.1]{HN11}. Hence we have $-(n-1)(n+1-d)<0$ which yields a contradiction. Thus 
the strict form of \eqref{almostdone} holds, this is equivalent to our statement.
\end{proof}


\begin{thebibliography}{ARVdV99}

\bibitem[ARVdV99]{ARV99}
E.~Amerik, M.~Rovinsky, and A.~Van~de Ven.
\newblock A boundedness theorem for morphisms between threefolds.
\newblock {\em Ann. Inst. Fourier (Grenoble)}, 49(2):405--415, 1999.

\bibitem[BCS04]{BCS04}
Jean-Yves Briend, Serge Cantat, and Mitsuhiro Shishikura.
\newblock Linearity of the exceptional set for maps of {$\bold P_k(\Bbb C)$}.
\newblock {\em Math. Ann.}, 330(1):39--43, 2004.

\bibitem[Bea01]{Bea01}
Arnaud Beauville.
\newblock Endomorphisms of hypersurfaces and other manifolds.
\newblock {\em Internat. Math. Res. Notices}, (1):53--58, 2001.

\bibitem[BH14]{a20}
Ama{\"e}l Broustet and Andreas H{\"o}ring.
\newblock Singularities of varieties admitting an endomorphism.
\newblock {\em Math. Ann.}, 360(1-2):439--456, 2014.

\bibitem[CLN00]{CL00}
D.~Cerveau and A.~Lins~Neto.
\newblock Hypersurfaces exceptionnelles des endomorphismes de {${\bf C}{\rm
  P}(n)$}.
\newblock {\em Bol. Soc. Brasil. Mat. (N.S.)}, 31(2):155--161, 2000.

\bibitem[Dol07]{Dol07}
Igor~V. Dolgachev.
\newblock Logarithmic sheaves attached to arrangements of hyperplanes.
\newblock {\em J. Math. Kyoto Univ.}, 47(1):35--64, 2007.

\bibitem[Fis01]{Fis01}
Gerd Fischer.
\newblock {\em Plane algebraic curves}, volume~15 of {\em Student Mathematical
  Library}.
\newblock American Mathematical Society, Providence, RI, 2001.
\newblock Translated from the 1994 German original by Leslie Kay.

\bibitem[FS94]{FS94}
John~Erik Forn{\ae}ss and Nessim Sibony.
\newblock Complex dynamics in higher dimension. {I}.
\newblock {\em Ast\'erisque}, (222):5, 201--231, 1994.
\newblock Complex analytic methods in dynamical systems (Rio de Janeiro, 1992).

\bibitem[HN11]{HN11}
Jun-Muk Hwang and Noboru Nakayama.
\newblock On endomorphisms of {F}ano manifolds of {P}icard number one.
\newblock {\em Pure Appl. Math. Q.}, 7(4, Special Issue: In memory of Eckart
  Viehweg):1407--1426, 2011.

\bibitem[NZ10]{NZ10}
Noboru Nakayama and De-Qi Zhang.
\newblock Polarized endomorphisms of complex normal varieties.
\newblock {\em Math. Ann.}, 346(4):991--1018, 2010.

\bibitem[PS89]{PS89}
K.~H. Paranjape and V.~Srinivas.
\newblock Self-maps of homogeneous spaces.
\newblock {\em Invent. Math.}, 98(2):425--444, 1989.

\bibitem[Sai80]{Sai80}
Kyoji Saito.
\newblock Theory of logarithmic differential forms and logarithmic vector
  fields.
\newblock {\em J. Fac. Sci. Univ. Tokyo Sect. IA Math.}, 27(2):265--291, 1980.

\bibitem[Zha13]{Zha13}
De-Qi Zhang.
\newblock Invariant hypersurfaces of endomorphisms of the projective 3-space.
\newblock In {\em Affine algebraic geometry}, pages 314--330. World Sci. Publ.,
  Hackensack, NJ, 2013.

\bibitem[Zha14]{Zha14}
De-Qi Zhang.
\newblock Invariant hypersurfaces of endomorphisms of projective varieties.
\newblock {\em Adv. Math.}, 252:185--203, 2014.

\end{thebibliography}

\def\cprime{$'$}

\end{document}